\documentclass{amsart}

\usepackage{hyperref}

\usepackage{amsmath, amssymb}
\usepackage{enumerate}

\newtheorem{theorem}{Theorem}[section]
\newtheorem{lemma}[theorem]{Lemma}

\newtheorem{dfn}[theorem]{Definition}

\numberwithin{equation}{section}

\newcommand{\vp}{\varphi}

\renewcommand{\Im}{\mathop{\rm Im}}
\renewcommand{\Re}{\mathop{\rm Re}}

\newcommand{\clos}{\mathop{\rm clos}}
\newcommand{\inte}{\mathop{\rm int}}

\newcommand{\beq}{\begin{equation}}
\newcommand{\eeq}{\end{equation}}
\newcommand{\bea}{\begin{eqnarray}}
\newcommand{\eea}{\end{eqnarray}}
\newcommand{\bal}{\begin{align}}
\newcommand{\eal}{\end{align}}
\newcommand{\nn}{\nonumber}
\newcommand{\ga}{\gamma}

\newcommand{\si}{\sigma}
\newcommand{\pa}{\partial}

\newcommand{\la}{\lambda}

\newcommand{\Z}{\mathbb{Z}}
\newcommand{\R}{\mathbb{R}}

\newcommand{\I}{\mathrm{i}}
\newcommand{\siul}{\sigma^{\mathrm{u,l}}}
\newcommand{\siu}{\sigma^{\mathrm{u}}}
\newcommand{\sil}{\sigma^{\mathrm{l}}}
\newcommand{\sipmu}{\sigma_\pm^{\mathrm{u}}}
\newcommand{\sipml}{\sigma_\pm^{\mathrm{l}}}
\newcommand{\sipmul}{\sigma_\pm^{\mathrm{u,l}}}
\newcommand{\simpul}{\sigma_\mp^{\mathrm{u,l}}}
\newcommand{\lau}{\lambda^{\mathrm{u}}}
\newcommand{\lal}{\lambda^{\mathrm{l}}}
\newcommand{\laul}{\lambda^{\mathrm{u,l}}}

\newcommand{\lz}{\ell^2(\Z)}
\newcommand{\liz}{\ell^1(\Z_\pm)}

\begin{document}

\title[Scattering Theory with Finite-Gap Backgrounds]{Scattering Theory with Finite-Gap Backgrounds: Transformation Operators and Characteristic Properties of Scattering Data}

\author[I. Egorova]{Iryna Egorova}
\address{Institute for Low Temperature Physics\\ 47,Lenin ave\\ 61103 Kharkiv\\ Ukraine}
\email{\href{mailto:iraegorova@gmail.com}{iraegorova@gmail.com}}

\author[J. Michor]{Johanna Michor}
\address{Faculty of Mathematics\\ University of Vienna\\
Nordbergstrasse 15\\ 1090 Wien\\ Austria\\ and International Erwin Schr\"odinger
Institute for Mathematical Physics\\ Boltzmanngasse 9\\ 1090 Wien\\ Austria}
\email{\href{mailto:Johanna.Michor@univie.ac.at}{Johanna.Michor@univie.ac.at}}
\urladdr{\href{http://www.mat.univie.ac.at/~jmichor/}{http://www.mat.univie.ac.at/\string~jmichor/}}

\author[G. Teschl]{Gerald Teschl}
\address{Faculty of Mathematics\\ University of Vienna\\
Nordbergstrasse 15\\ 1090 Wien\\ Austria\\ and International Erwin Schr\"odinger
Institute for Mathematical Physics\\ Boltzmanngasse 9\\ 1090 Wien\\ Austria}
\email{\href{mailto:Gerald.Teschl@univie.ac.at}{Gerald.Teschl@univie.ac.at}}
\urladdr{\href{http://www.mat.univie.ac.at/~gerald/}{http://www.mat.univie.ac.at/\string~gerald/}}

\thanks{Research supported by the Austrian Science Fund (FWF) under Grants No.\ Y330 and V120.}
\thanks{Math. Phys. Anal. Geom. {\bf 16}, 111--136 (2013)}

\begin{abstract}
We develop direct and inverse scattering theory for Jacobi operators (doubly infinite second order
difference operators) with steplike coefficients which are asymptotically close to different finite-gap
quasi-periodic coefficients on different sides. We give necessary and sufficient conditions for the
scattering data in the case of perturbations with finite second (or higher) moment.
\end{abstract}

\maketitle

\section{Introduction}

The present paper is concerned with inverse scattering theory for Jacobi operators (doubly infinite second order
difference operators) given by
\beq\label{def11}
H f(n)=a(n-1) f(n-1) + b(n) f(n)+ a(n) f(n+1),
\quad a(n)>0,\ b(n)\in\mathbb R,
\eeq
and acting on the Hilbert space $\ell^2(\Z)$ of doubly infinite square summable sequences.
We will consider this problem in the case where the coefficients of $H$ are asympto\-tically
close to the coefficients of two (in general different) finite-gap operators
 \beq
\label{JdefHpm} H^\pm f(n) = a^\pm(n)
f(n+1) + a^\pm(n-1) f(n-1) + b^\pm(n) f(n)
\eeq
as $n\to\pm\infty$\footnote{Here, the signs "$-$" and "$+$"
refer to the left and right half-axis, respectively.}. By a finite-gap operator we mean an operator associated with quasi-periodic
sequences whose spectrum consists of finitely many bands. In particular, this covers the case of
periodic or constant Jacobi operators.

\begin{dfn}
Let $q=0,1,2,\dots$ be a nonnegative integer. We say that $H$ has the $q$'th moment finite
with respect to $H^+$, $H^-$ if its coefficients satisfy
\beq \label{Jhypo}
\sum_{n = 0}^{\pm
\infty} |n|^q \Big(|a(n) - a^\pm(n)| + |b(n) - b^\pm(n)| \Big) < \infty.
\eeq
The set of all Jacobi operators $H$ satisfying \eqref{Jhypo} will be denoted by $\mathcal B_q(H^+,H^-)$.
\end{dfn}

Given two fixed limiting (or background) operators $H^\pm$ as above, the aim is to find necessary and sufficient conditions on the scattering data to
correspond to some operator $H$ in the class $\mathcal B_q(H^+,H^-)$.

In fact, the associated problem for the case $q=1$ is well studied (see \cite{emtqps}, \cite{emtstep}, \cite{emtstp2}, \cite{Khan}, \cite{Khan1}, \cite{voyu}).
However, as pointed out to us by A.\ K.\ Khanmamedov, the proof of Lemma~7.3 from \cite{emtqps} contains an error which consequently also affects our
follow-up papers \cite{emtstep} and \cite{emtstp2}.
 One aim of the present paper is to correct this error. Unfortunately, the properties of scattering data obtained below are not sufficient
to solve the inverse scattering problem in $\mathcal B_1(H^+,H^-)$, but they are sufficient to reconstruct a unique operator
$H$ which belongs to the class $\mathcal B_0(H^+,H^-)$. For $q\ge 2 $ this mismatch does not arise and
we give necessary and sufficient conditions for the scattering data to correspond to an operator in the class $\mathcal B_q(H^+,H^-)$.

We emphasize that the mismatch for $q=1$ only occurs if the limiting operators are genuine finite-gap operators and does
not occur in the case of constant limiting operators. In particular, the approach used here allows us to solve the inverse scattering
problem in the class $\mathcal B_1(H^+_{const},H^-_{const})$ as a special case, when the background
operators are constant and form a simple step,
\beq \label{Jconstb}
H^\pm_{const} f(n) = a^\pm f(n+1) + a^\pm f(n-1) + b^\pm f(n), \quad
a^\pm>0.
\eeq
Moreover, in Section~\ref{Contin} we show that in the case $H^+_{const}=H^-_{const}$ the reflection coefficient is continuous at the edges of the spectrum in the
resonance case for the finite first moment. Note, that for an arbitrary finite gap background $H^+=H^-$ the continuity of the reflection coefficient at the edges of the spectrum for $q=1$ remains an open question. The only case when this
fact was established  corresponds to a special type of the Jacobi operator (see \cite{Kh5}), which is associated with the Volterra (or Kac--van Moerbeke) lattice
\beq
b(n)\equiv 0,\quad \sum |n|(|a(2n) - A_1| + |a(2n+1)-A_2|)<\infty, \quad A_1\neq A_2.
\eeq
 
Roughly, the scheme for solving the direct scattering problem for the case of a fixed
finite moment $q$ is the following:

 {\it Step 1.}  Construct transformation operators $K_\pm$ associated with both sides. Those operators
intertwine between $H$ and $H^\pm$ in the sense that $H K_\pm = K_\pm H^\pm$, respectively.
Show that the corresponding kernels $K_\pm(n,m)$ are triangular.

 {\it Step 2.} (a) Relate decay properties of the kernels of the transformation operators with $q$.
(b) Derive formulas connecting these kernels with the coefficients $a(n), b(n)$.
(c) Investigate the analytical properties of the Jost solutions with respect to $q$.

 {\it Step 3.}
Investigate the spectrum of $H$; derive the properties of the scattering matrix,
in particular, its unitary property and asymptotical behavior at infinity; establish the connection
between left and right scattering data.

 {\it Step 4.}
Derive the left and right Marchenko equations, which are the main equations of the inverse problem and connect the transformation operators to the set $\mathcal S(H)$ of scattering data (consisting of the
scattering matrix, the discrete spectrum, and normalizing constants).  The kernels $F_\pm$ of the Marchenko equations only depend on $\mathcal S(H)$.

 {\it Step 5.} Establish decay properties of $F_\pm$ in terms of $q$.

The solution of the inverse scattering problem consists of the following steps:

 {\it Step 6.} (a) Now we are given a set $\mathcal S$ of the same structure as the set of scattering data and
with properties as described in Steps 3 and 5. This implies that we have given functions $F_\pm$
which satisfy the prescribed decay properties. These properties are sufficient to prove that
the left and right Marchenko equations are uniquely solvable with respect to the kernels of the
transformation operators. Invoking Step~2, (b), the kernels give rise to two sets of coefficients.
(b) Both restored coefficients can in general be well controlled on the associated half-axis; on the
opposite half-axis, they cannot be analyzed. To ensure that the solution lies within the class, one
has to show that both sets of coefficients have the $q$'th moment finite on the associated side.

  {\it Step 7.} This is technically the most challenging step. We prove a uniqueness theorem,
that is, we show that the two restored sets of coefficients are equal and give rise to a unique
Jacobi operator $H$ with $\mathcal S= \mathcal S(H)$.

The pioneering work for this scattering problem was done by Kay and Moses \cite{KM} and mathematically
rigorously by Faddeev \cite{Fad}. They focussed on studying the main analytical properties of the scattering
data (Step~3), on the derivation of the Marchenko equation (Step~4), and on proving its unique solvability (Step 6, (a)). Steps 1--2 and Step 6, (b), were
already studied by Marchenko \cite{M2} when solving the scattering problem on the half-axis.
However, later investigations by Deift and Trubowitz \cite{DT} showed that the properties of the scattering data listed in \cite{Fad} are not sufficient to prove the uniqueness theorem (Step 7). An example for a set
$\mathcal S$ was given there, which satisfied all conditions proposed in \cite{Fad}, but the restored
potential did not have the required decaying behavior on one half-axis. It turned out that the behavior of the scattering data at the edge
of the continuous spectrum plays a key role in the proof of the uniqueness theorem. For finite
second moments of perturbation this behavior was described in \cite{DT}. There the characteristic properties of the scattering data were derived and the direct/inverse scattering problems were solved in the respective class of Schr\"odinger operators. But the approach used in \cite{DT} was not applicable for  finite first moments. One has to emphasize, that a description of the characteristic properties of the scattering data for a class of potentials with finite first moments (these constitute the largest class of potentials for which direct/inverse scattering can be studied within the class) is a much more complicated problem.
It was solved by Marchenko in 1977, \cite{mar}. The condition on the
scattering data at the edge of the continuous spectrum is now referred to as Marchenko condition (see also the discussion before Theorem~\ref{Jtheor2}).
Marchenko's approach for solving the inverse scattering problem became the classic method
and was successfully generalized to several other types of operators and potentials, such as asymptotically periodic,
finite-gap non-periodic, or steplike potentials.

In this note we study some features of Marchenko's method applied to Jacobi operators on finite-gap backgrounds.
Scattering theory for Jacobi operators with constant background is a likewise classic
topic with a long tradition. Originally developed on an informal level by Kac, Case, and Geronimo \cite{CK}, \cite{dinv4}, \cite{GC},
the first rigorous results were given by Guseinov \cite{gu} with further extensions by Teschl \cite{tivp}, \cite{tjac}.
Moreover, these results have direct applications for solving the Toda lattice via the inverse scattering transform \cite{tjac} and
investigating its long-time asymptotics via the nonlinear steepest decent method (see, e.g., the review \cite{KrT}).
For the case of steplike backgrounds we refer to \cite{BE}, \cite{eg} and for applications to the
Toda lattice to \cite{bdme2}, \cite{bdmek}, \cite{dkkz}, \cite{vdo}. The case of (steplike)
periodic backgrounds was considered in \cite{emtqps}, \cite{emtstep}, \cite{emtstp2}, \cite{Khan}, \cite{Khan1}, \cite{voyu} and applications to the Toda lattice can be found in
 \cite{emtist}, \cite{emtsr}, \cite{kateptr}, \cite{katept}, \cite{KrT2}, and \cite{mtqptr}.

\section{Direct scattering problem}

{\it Steps 1--2.} (Estimates on the transformation operators, properties of the Jost solutions, etc.).
In the case of constant background operators, in particular, when they coincide with the discrete
Laplacian $H^+_{const}=H^-_{const}=H_0$, where $H_0y(n)=\frac 1 2 \{y(n-1) + y(n+1)\}$, the transformation operators convert the "exponents"  (the generalized eigenfunctions of $H_0$)
into the Jost solutions. Namely, if $\la\in\mathbb C$ is a spectral parameter, then these eigenfunctions are $z^{\pm n}$, where $\la=\frac 1 2 (z+z^{-1})$, $|z|\leq 1$. When the background operator
is a finite-gap operator, the role of "exponents" is played by its Weyl solutions.
Though the behavior of the finite-gap Weyl solutions remains exponential as $\la\to \infty$, they exhibit
a much more complicated structure than the simple exponents.

We first recall some basic facts on the spectral analysis of finite-gap Jacobi operators\footnote{For more details see Chapter 9 of \cite{tjac}.}.
Since we consider two different background operators, we introduce notations for both of them simultaneously.
Let the spectra of $H^\pm$ in \eqref{JdefHpm} consist of $r_\pm+1$ bands,
$$
\si_\pm:=\si(H^\pm) =
\bigcup_{j=0}^{r_\pm} [E_{2j}^\pm,E_{2j+1}^\pm],\quad
E_0^\pm < E_1^\pm < \cdots < E_{2r_\pm+1}^\pm.
$$
Associated with these sets are the Riemann surfaces of the functions
\begin{equation}\label{JdefP}
P_\pm(\la)= -\prod_{j=0}^{2r_\pm+1} \sqrt{\la-E_j^\pm},
\end{equation}
where $\sqrt{.}$ denotes the standard branch of the square root cut along $(-\infty,0)$.
Each of the operators $H^\pm$ is uniquely defined by its spectrum $\si_\pm$ and the divisor of poles of its Weyl function (as a function on the respective Riemann surface)
\beq\label{divW}\sum_{j=1}^{r^\pm}(\mu_j^\pm, \si_j^\pm),\ \mbox{where}\ \
\mu_j^\pm\in[E_{2j-1}^\pm,E_{2j}^\pm] \ \mbox{and}\  \si_j^\pm\in\{-1, 1\}.\eeq
It is more convenient to consider the spectral parameter in the complex plane with cuts than working on
two different Riemann surfaces. Therefore we identify $\mathbb C\setminus\si(H^+)$
with the upper sheet of the Riemann surface of $P_+$ and $\mathbb C\setminus\si(H^-)$ with the lower sheet
of the Riemann surface of $P_-$.
To distinguish the boundaries we denote the upper and lower points of the cuts along
$\sigma_\pm$ by $\sipmu$ and $\sipml$. Symmetric points of these cuts are denoted by
$\lau$ and $\lal$. In particular, for a function
$f$ we have
$$
f(\lau) :=
\lim_{\varepsilon\downarrow0} f(\lambda+\I\varepsilon), \quad
f(\lal) := \lim_{\varepsilon\downarrow0} f(\lambda-\I\varepsilon),
\quad \lambda\in\sigma_\pm.
$$

Let $\psi^\pm(\la,n)$ be the Weyl solutions of the spectral equations
\beq\label{J0.12}
H^\pm \psi(\la,n)=\la \psi(\la,n), \quad \la\in\mathbb C, \quad n\in \Z,
\eeq
which are uniquely defined by the conditions $\psi^\pm(\la,0)=1$ and
$\psi^\pm(\la,\cdot)\in\ell^2(\Z_\pm)$ as $\la\in\mathbb C\setminus\si_\pm$. Then the solution
$\psi^+(\la)$ (resp.\ $\psi^-(\la)$) coincides with the upper (resp.\ lower) branch of the Baker-Akhiezer function of the operator $H^+$ (resp.\ $H^-$). The twin branches of these functions are denoted by
 $\breve \psi^\pm(\la)$. They are solutions of \eqref{J0.12} as well, and satisfy $\breve\psi^\pm(\la,0)=1$
and  $\breve \psi^\pm(\la,\cdot)\in\ell^2(\Z_\mp)$ as
$\la\in\mathbb C\setminus\si_\pm$.
The solutions $\psi^\pm(\la,n)$, $\breve\psi^\pm(\la,n)$ are continuous till the boundary for
 $\la\to\laul\in\siul_\pm\setminus\partial\si_\pm$ as functions of the spectral parameter $\la$, where we denote
the set of spectral edges by $\partial\si_\pm:=\{E_0^\pm,\dots,E_{2r_\pm+1}^\pm\}$.
Moreover, they satisfy the symmetry condition
\beq \label{J0.6}
\psi^\pm(\lal,n)=
\overline{\psi^\pm(\lau,n)}=\breve\psi^\pm(\lau,n), \quad \la\in\si_\pm, \quad n\in \Z.
\eeq
To describe possible poles and other singularities of the Weyl solutions
we distinguish the following disjoint subsets of $\{\mu_1^\pm,\dots,\mu_{r_\pm}^\pm\}$
(compare \eqref{divW})
\begin{align}
\begin{split}
M^\pm &= \{\mu_j^\pm \,|\, \mu_j^\pm \in \R\setminus\si_\pm \mbox{
is a pole of } \psi^\pm(\la,1)\},\\
\hat{M}^\pm &= \{ \mu_j^\pm \,|\, \mu_j^\pm \in \partial \si_\pm \},
\end{split}
\end{align}
and introduce auxiliary functions by
\beq\label{0.63}
\delta_\pm(\la) := \prod_{\mu^\pm_j\in M_\pm}  (\la - \mu^\pm_j),\quad
\hat \delta_\pm(\la) := \delta_\pm(\la) \prod_{\mu^\pm_j \in \hat M_\pm} \sqrt{\la -
\mu^\pm_j},
\eeq
where $\prod=1$ if there are no multipliers. Accordingly, we set
\beq\label{deftilde}
\tilde\psi^\pm(\la,n)=\delta_\pm(\la) \psi^\pm(\la,n),\quad
\hat\psi^\pm(\la,n)=\hat\delta_\pm(\la) \psi^\pm(\la,n). \eeq
Then $\tilde\psi^\pm(\la,n)$ have no poles in open gaps of the spectra, satisfy the symmetry property \eqref{J0.6}, and have square root singularities at the points of the sets $\hat{M}^\pm$, whereas  $\hat\psi^\pm(\la,n)$ are continuous on $\mathbb C\setminus\si_\pm$ till the boundaries,
but in general violate the symmetry property.

Introduce the discrete Wronskians of the Weyl solutions by
\beq \label{new11}
W^\pm(\la):=a^\pm(n) \big(\breve\psi^\pm(\la,n) \psi^\pm(\la,n+1) -
\breve\psi^\pm(\la,n+1)\psi^\pm(\la,n)\big).
\eeq
Then
\beq\label{J2.9}
W^\pm(\la) =
\pm\frac{1}{g_\pm(\la)},\  \mbox{ where}\ \
g_\pm(\la):=\frac{\prod_{j=1}^{r_\pm}(\la-\mu_j^\pm)}{ P_\pm(\la)}
\eeq
are the Green functions of the operators  $H^\pm$ on the main diagonal at point $(0,0)$.
By the choice of the square root branch in $P_\pm(\la)$ the Green functions satisfy
\beq
\label{J0.10} \Im(g_\pm(\lau))>0, \quad \Im(g_\pm(\lal))<0, \quad
\la\in\si_\pm.
\eeq

The functions  $\psi^\pm(\la,\cdot)$ form complete orthogonal systems on the spectra $\si_\pm$ with
respects to the weights
 \beq \label{J1.121}
d\omega_\pm(\la) =
\frac{1}{2\pi\I} g_\pm(\la) d\la. \eeq Namely,
\begin{equation}
\oint_{\si_\pm}\overline{\psi^\pm(\la,m)}\psi^\pm(\la,n)
d\omega_\pm(\la) = \delta(n,m), \label{J1.14}
\end{equation}
where $\delta(n,m)$ is the Kronecker symbol and
\beq \label{1.141}
\oint_{\si_\pm}f(\la) d\la := \int_{\sipmu} f(\lau) d\la - \int_{\sipml} f(\lal) d\la.
\eeq

Next, we describe the properties of the transformation operators on finite-gap backgrounds, following \cite{Khan} and \cite{emtqps} with a small correction in the steplike case. Assume that the
perturbation coefficients $\{a(n)-a^\pm(n), b(n)-b^\pm(n)\}$ have minimal (i.e.\ first) finite moments
on the half-axes. Evidently, all remains valid for finite higher moments.

 \begin{lemma}\label{Jprj}
Let $H \in \mathcal B_1(H^+, H^-)$. Then there exist Jost solutions $\phi_\pm(\la,n)$ of the spectral problem
\beq\label{speceq}
H\phi(\la,n)=\la\phi(\la,n)
\eeq
which are asymptotically close to the Weyl solutions $\psi^\pm(\la,n)$ of the background operators
$H^\pm$ as $n\to\pm\infty$. These solutions can be represented as
\beq \label{J3.16}
\phi_\pm(\la,n) = \sum_{m=n}^{\pm \infty}
K_\pm(n,m) \psi^\pm(\la, m), \ \ \la\in\mathbb C,
\eeq
where the kernels $K_\pm(n,\cdot)$ of the transformation operators are real valued and satisfy for
$\pm m > \pm n$
 \beq \label{chto}
\left|K_\pm(n,m)\right| \leq C_\pm(n)\ \prod_{j=n-{\scriptstyle{0 \atop 1}}}^{\pm\infty}\frac{a(j)}{a^\pm(j)}\
\sum_{j=[\frac{m+n}{2}] }^{\pm \infty} \Big(|a(j)-a^\pm(j)| + |b(j)-b^\pm(j)|\Big).
\eeq
The functions $C_\pm(n) > 1$ decay as $n\rightarrow \pm \infty$.
\end{lemma}

Note that in the steplike case, $C_\pm(n)$ grow with the order of the products in
\eqref{chto} as $n\to\mp\infty$. In the case $H^+=H^-$ they remain bounded and are usually replaced by constants.
From \eqref{J3.16} and \eqref{J1.14} it follows that
$$
K_\pm(n,m)=\oint_{\si_\pm}\phi_\pm(\la,n)
\breve\psi_\pm(\la,m)d\omega_\pm(\la), \quad \pm m\geq \pm n,
$$
and $K_\pm(n,m)=0$ for $\pm m<\pm n$. Therefore,
\begin{align}\label{Jsmes}
\begin{split}
& a(n-1)K_\pm(n-1,m) + b(n) K_\pm(n,m) + a(n) K_\pm(n+1,m) \\
& = a^\pm(m-1)K_\pm(n,m-1) + b^\pm(m) K_\pm(n,m) + a^\pm(m) K_\pm(n,m+1).
\end{split}
\end{align}
Evaluating \eqref{Jsmes} at $m=n\mp 1$ and $m=n$ we obtain
\begin{align} \label{J0.13}
a(n) &= a^+(n)\frac{K_+(n+1, n+1)}{K_+(n,n)} = a^-(n)\frac{K_-(n, n)}{K_-(n+1,n+1)}, \\
\begin{split}
b(n) &= b^+(n) + a^+(n)\frac{K_+(n, n+1)}{K_+(n,n)} -
a^+(n-1) \frac{K_+(n-1, n)}{K_+(n-1,n-1)}, \\
b(n) &= b^-(n) + a^-(n-1)\frac{K_-(n, n-1)}{K_-(n,n)} - a^-(n)
\frac{K_-(n+1, n)}{K_-(n+1,n+1)}.
\end{split}
\end{align}
With these formulas we can compute the coefficients of the perturbed operator $H$ via the kernels
of the left and right transformation operators.

For $\la \in \siu_\pm\cup\sil_\pm$, equations \eqref{speceq} also have solutions
$$
\breve\phi_\pm(\la,n)=\sum_{m=n}^{\pm \infty} K_\pm(n,m)
\breve\psi^\pm(\la, m), \quad \la \in \siu_\pm\cup\sil_\pm,
$$
which cannot be continued to the complex plane under condition \eqref{Jhypo} for fixed $q$.
For $\la \in \si_\pm$, we have $\breve\phi_\pm(\la,n)=\overline{\phi_\pm(\la,n)}$,
and \eqref{Jhypo}, \eqref{J3.16}, and \eqref{new11} imply
\beq\label{J0.62}
W(\overline{\phi_\pm(\la)}, \phi_\pm(\la)) =
W^\pm(\la) = \pm g_\pm(\la)^{-1}.
\eeq

For any $q=1,2,\dots$, due to \eqref{J3.16} the Jost solutions $\phi_\pm$ inherit almost all properties of their background Weyl solutions, in particular, the structure of poles and square root singularities. But at the edges of the spectra $\pa\si_\pm$ the influence of the moment $q$ is already perceptible. For $q=2,3,\dots$, the Jost solutions are differentiable functions of the local parameter $\sqrt{\la-E}$ at $E\in\pa\si_\pm$, but for $q=1$ they cease to be.
This  complicates investigating the scattering data at the edges of the spectra.

We turn to {\it Step 3} next.
Define in analogy to \eqref{0.63}, \eqref{deftilde}
\beq\label{J0.64}
\tilde\phi_\pm(\la,n)=\delta_\pm(\la)
\phi_\pm(\la,n),\quad \hat\phi_\pm(\la,n)=\hat\delta_\pm(\la)
\phi_\pm(\la,n),
\eeq
and introduce the sets
\beq
\label{J2.5}
\si:=\si_+\cup\si_-, \quad
\si^{(2)}:=\si_+\cap\si_-, \quad
\si_\pm^{(1)}=\clos \big(\si_\pm\setminus\si^{(2)}\big).
\eeq
Then $\si$ is the absolutely continuous spectrum of $H$ and $\si_+^{(1)}\cup \si_-^{(1)}$
and $\si^{(2)}$ are the spectra of multiplicity one and two (cf.\ \cite{emtstp2}). Denote by $\inte(\si)$ the sets of inner points of the spectrum, $\inte(\si):=\si\setminus\pa\si$. In addition to the continuous spectrum,  $H$ has a finite number of eigenvalues
$\si_d=\{\la_1,\dots,\la_p\} \subset \R \setminus \si$ (finiteness of $\si_d$ is proven in \cite{tosc}).
We introduce normalizing constants by
\beq \label{Jnorming}
\ga_{\pm, k}^{-1}=\sum_{n \in
\Z}|\tilde\phi_\pm(\la_k, n)|^2, \quad 1 \leq k \leq p.
\eeq
The Wronskian of two Jost solutions, given by
\beq\label{J2.8}
W(\la):=a(n)
\left(\phi_-(n)\phi_+(n+1) - \phi_-(n+1)\phi_+(n)\right),
\eeq
is meromorphic on $\mathbb{C}\setminus\si$, since the Jost solutions are meromorphic there,
with possible poles on $M_+\cup M_-\cup(\hat M_+\cap\hat M_-)$ and possible square
root singularities on $\hat M_+\cup\hat M_-\setminus (\hat M_+\cap \hat M_-)$.
Consider the scattering relations
\beq\label{J6.16}
T_\mp(\la)
\phi_\pm(\la,n) =\overline{\phi_\mp(\la,n)} +
R_\mp(\la)\phi_\mp(\la,n), \quad\la\in\simpul.
\eeq
For steplike cases, the entries of the scattering matrix generally live on different sets.
The left reflection coefficient $R_-$ and transmission coefficient $T_-$ are defined on the upper and lower sides of $\si_-$, the right coefficients $R_+$ and $T_+$ are defined on the sides of $\si_+$.
The scattering data $\mathcal S$ has the following structure,
 \begin{align}\label{J4.6}
\begin{split}
{\mathcal S} &= \big\{ R_+(\la),\,T_+(\la),\,
\la\in\si_+^{\mathrm{u,l}}; \,
R_-(\la),\,T_-(\la),\, \la\in\si_-^{\mathrm{u,l}};\\
& \qquad \la_1,\dots, \la_p \in\mathbb{R} \setminus
(\si_+\cup\si_-),\, \ga_{\pm, 1}, \dots, \ga_{\pm, p}
\in\mathbb{R}_+\big\}.
\end{split}
\end{align}
The entries of $\mathcal S$ are not independent of each other, their dependencies are collected in the
next lemma (see \cite{emtstp2}). These results hold for any fixed $q=1,2,\dots$ and might be specified
for $q=2,3,\dots$, which is not of interest here.
Since the conditions in this list are not independent, some of them are partly covered by others. The dependence of $\mathcal S$ on $q$ will be expressed by property ${\bf IV_q}$ in Lemma~\ref{Jlem4.1} below.

\begin{lemma} \label{Jlem2.3}
Let $H\in\mathcal B_q(H^+,H^-)$ for fixed $q=1,2,\dots$.
Then the entries of the scattering matrix of $H$ have the following properties

${\bf I}$.
$$
\begin{array}{lll}
{\bf (a)}&T_\pm(\lau) =\overline{T_\pm(\lal)},\ R_\pm(\lau) =\overline{R_\pm(\lal)}& \la\in\si_\pm, \nn \\
{\bf (b)}& T_\pm(\la)\overline{T_\pm^{-1}(\la)}
 = R_\pm(\la),&
\la\in\si_\pm^{(1)},\nn\\
{\bf (c)}& 1 - |R_\pm(\la)|^2 =
g_\pm(\la)g_\mp^{-1}(\la)\,|T_\pm(\la)|^2,
&\la\in\si^{(2)},\\
{\bf (d)}&\overline{R_\pm(\la)}T_\pm(\la) +
R_\mp(\la)\overline{T_\pm(\la)}=0,& \la\in\si^{(2)}.
\end{array}
$$

${\bf II}$. The functions $T_\pm(\la)$ can be continued analytically in
 $\mathbb{C} \setminus (\si\cup M_\pm\cup\breve
M_\pm\cup\si_d)$, where they satisfy
\beq \label{J2.18}
\frac{1}{T_+(\la) g_+(\la)} = \frac{1}{T_-(\la)
g_-(\la)} =:W(\la).
\eeq
The function $W(\la)$ has the following properties:
\begin{enumerate}[{\bf (a)}]
\item
$\tilde W(\la) = \delta_+(\la)\delta_-(\la)W(\la)$ is holomorphic in $\mathbb{C}\setminus\si$ with
simple zeros at $\la_k$,
\beq \label{J2.11}
\Big(\frac{ d\tilde W}{d \la}(\la_k)\Big)^2
=\ga_{+,k}^{-1}\ga_{-,k}^{-1}.
\eeq
In addition, $\overline{\tilde W(\lau)}=\tilde W(\lal)$ as $ \la\in\si$ and
$\tilde W(\la)\in\mathbb{R}$ as $ \la\in\mathbb{R}\setminus \sigma$.

\item $\hat W(\la) = \hat\delta_+(\la) \hat\delta_-(\la)
W(\la)$ is continuous on $\mathbb{C}\setminus\si$ up to the boundary $\siu\cup\sil$.
It can have zeros on $\Sigma_v:=\pa\si\cup(\pa\si_+^{(1)}
\cap\pa\si_-^{(1)})$\footnote{This is the set of possible virtual levels of $H$.} and does not vanish at any other point of $\si$. If $\hat W(E)=0$ for
$E\in\Sigma_v$, then
\begin{align} \label{Jestim}
\hat W^{-1}(\la) &= O\big((\la -
E)^{-1/2}\big), \quad\mbox{for $\la \in \si$ close to $E$}, \\
 \label{Jestim2} \hat W^{-1}(\la) &=
O\big((\la-E)^{-1/2-\varepsilon}\big), \quad \mbox{for
$\la\in\mathbb C\setminus\si$ close to $E$}.
\end{align}
\item $\, T_+(\infty)=T_-(\infty)>0$.
\end{enumerate}

${\bf III}$. \begin{enumerate}[{\bf (a)}]
 \item The reflection coefficients $R_\pm(\la)$ are continuous functions on $\inte(\sipmul)$. \\ If
$E\in\pa\si_+\cup \partial\si_-$ and $\hat W(E)\neq 0$, then
$R_\pm(\la)$ are also continuous at $E$.
\item If
$E\in\pa\si_+\cap \partial\si_-$ and $\hat W(E)\neq 0$, then
\beq\label{JP.1}
R_\pm(E) = \left\{
\begin{array}{c@{\quad\mbox{as}\quad}l}
-1 & E\notin\hat M_\pm,\\
1 & E\in\hat M_\pm. \end{array}\right.
\eeq
\end{enumerate}
\end{lemma}

\vskip 0.2cm
{\it Remarks.}
\begin{itemize}
\item For any fixed $q=2,3,\dots$, properties \eqref{Jestim} and \eqref{Jestim2} can be replaced by
(see \cite{BET}):
{\it If $\hat W(E)=0$ for
    $E\in\Sigma_v$, then $\hat W(\la)= C\sqrt{\la - E}(1+o(1))$, $C\neq 0$, for $\la\to E$.}
\item For any fixed $q=2,3,\dots$, property ${\bf III}$, ${\bf (a)}$ is given by:
{\it The reflection coefficients $R_\pm(\la)$ are continuous functions on $\sipmul$.}
\end{itemize}

\vskip 0.2cm
{\it Steps 4--5} (The Marchenko equations and estimates on their kernels)

\begin{theorem}{\rm (\cite{emtstp2})}
Let $H\in\mathcal B_q(H^+,H^-)$ for a fixed
$q=1,2,\dots$. Then the kernels of the transformation operators satisfy the Marchenko equations
\beq \label{Jglm1 q}
K_\pm(n,m) + \sum_{\ell=n}^{\pm \infty}K_\pm(n,\ell)F_\pm(\ell,m) =
\frac{\delta(n,m)}{K_\pm(n,n)}, \quad \pm m \geq \pm n,
\eeq
where
\begin{align}\label{kerM}
\begin{split}
&F_\pm(m,n) = \oint_{\si_\pm} R_\pm(\la) \psi^\pm(\la,m)
\psi^\pm(\la,n) d\omega_\pm \\
& + \int_{\si_\mp^{(1),u}}
|T_\mp(\la)|^2 \psi^\pm(\la,m) \psi^\pm(\la,n) d\omega_\mp +
\sum_{k=1}^p \ga_{\pm,k} \tilde\psi^\pm(\la_k,n)
\tilde\psi^\pm(\la_k,m).
\end{split}
\end{align}
\end{theorem}
We abbreviate
\beq\label{Jkappa}
\kappa_\pm(n,m)=\frac{K_\pm(n,m)}{K_\pm(n,n)},\quad \pm m>\pm n,
\eeq
and rewrite the Marchenko equations \eqref{Jglm1 q} by
 \begin{align}\label{Jmaro}
& \kappa_\pm(n,m) +F_\pm(n,m) +
 \sum_{\ell=n\pm 1}^{\pm\infty} \kappa_\pm(n,\ell)F_\pm(\ell,m)=0,\quad
 \pm m>\pm n, \\ \label{Jmarn}
& 1+F_\pm(n,n)+\sum_{\ell=n\pm 1}^{\pm\infty}
\kappa_\pm(n,\ell)F_\pm(\ell,n)=\frac{1}{K_\pm(n,n)^2}.
\end{align}

The last remaining property to complete the list of necessary and
sufficient conditions on the scattering data depends on $q$.
\begin{lemma}\label{Jlem4.1}
Let $H\in\mathcal B_q(H^+,H^-)$ for a fixed  $q=1,2,\dots$. Then

${\bf IV_q}.$ The functions $F_\pm(n,m)$ defined by \eqref{kerM} satisfy the following conditions.
\begin{enumerate}[\bf (i)]
\item There exist functions $p_\pm(n)\geq 0$, $C_\pm(n)>0$ such that
$|n|^q p_\pm(n) \in \ell^1(\Z_\pm)$, $C_\pm(n)$
do not increase as $n\to\pm\infty$, and such that
\beq \label{Jfourest}
|F_{\pm}(n, m)|\leq C_\pm(n)\sum_{j=n+m}^{\pm \infty} p_\pm(j).
\eeq
\item The following estimates hold
\begin{align}\label{Jfourest1}
& \sum_{n = n_0}^{\pm \infty}
|n|^{\alpha} \big| F_{\pm}(n,n) -
 F_{\pm}(n \pm 1, n \pm 1)\big| < \infty,    \\ \label{Jdif1}
& \sum_{n = n_0}^{\pm \infty}|n|^{\alpha} \big|a^\pm(n) F_{\pm}(n,n+1)
- a^\pm(n-1)  F_{\pm}(n - 1, n)\big| < \infty,
\end{align}
where $\alpha=q$ as $q=2,3,\dots$ and $\alpha=0$ as $q=1$.
\end{enumerate}
\end{lemma}

\begin{proof} {\bf (i)} The considerations for the left and right half axes are analogous,
so we prove \eqref{Jfourest} for the "$+$" case and drop the index.
For $m>n$, we set in \eqref{Jmaro}
$\kappa:=\kappa_+$, $F:=F_+$,
and let
$\hat a:=a^+$, $\hat b:=b^+$, and $C(n)=C_+(n)$,
where $C_+(n)\geq 1$
is the function from  estimate \eqref{chto}. Then
\beq\label{Jdef1J}
|\kappa(n,m)|< \sigma(n+m),
\eeq
where
\beq\label{Jest1J}
\sigma(n):=C(n)
\sum_{\ell=\left[\frac{n}{2}\right]}^\infty\big(|a(\ell) - \hat a(\ell)|+ |b(\ell)
- \hat b(\ell)|\big).
\eeq
Define $\sigma_1(n):=\sum_{\ell=n}^\infty
\sigma(\ell)$, then
$$
n^{q-1}\sigma(n)\in\ell^1(\mathbb Z_+ ),\quad q=1,2,\dots, \quad
n^{q-2}\sigma_1(n)\in\ell^1(\mathbb Z_+ ),\quad q=2,3,\dots.
$$
Moreover,
\beq\label{JneqJ}
\sigma_1(n)\geq\sigma_1(\ell), \quad
\sigma(n)\geq\sigma(\ell), \quad n<\ell,
\eeq
and since
$$\frac{\sigma_1^s(n)-\sigma_1^s(n+1)}{s!}=(\sigma_1(n)
-\sigma_1(n+1))\frac{\sigma_1^{s-1}(n)+\dots+\sigma_1^{s-1}(n+1)}{s!},$$
then
\beq\label{JsigmaJ}
\frac{\sigma_1^s(n)-\sigma_1^s(n+1)}{s!}\geq(\sigma_1(n)
-\sigma_1(n+1))\frac{\sigma_1^{s-1}(n+1)}{(s-1)!}.
\eeq

To obtain the required estimate, we apply the method of successive approximations to \eqref{Jmaro}. Let
$$
F^{(0)}(n,m)=-\kappa(n,m),\quad F^{(s)}(n,m)=
-\sum_{\ell=n+1}^\infty \kappa(n,\ell)F^{(s-1)}(\ell,m).
$$
We prove by induction that
\beq\label{Jest3J}
|F^{(s)}(n,m)|\leq \sigma(n+m)\frac{\sigma_1^{s}(2n+1)}{s!}.
\eeq
By \eqref{Jdef1J} and \eqref{Jest1J}, this estimate is true for $s=0$.
For $s\geq 1$, take into account  \eqref{JneqJ} and \eqref{JsigmaJ}, then
\begin{align*}
|F^{(s)}(n,m)|&\leq \sum_{\ell=n+1}^\infty \sigma(n+\ell)\sigma(\ell+m)
\frac{\sigma_1^{s-1}(2\ell+1)}{(s-1)!}\\
&\leq \sigma(n+m) \sum_{\ell=n+1}^\infty
\sigma(n+\ell)\frac{\sigma_1^{s-1}(n+\ell+1)}{(s-1)!}\\
&=\sigma(n+m) \sum_{\ell=n+1}^\infty
(\sigma_1(n+\ell)-\sigma_1(n+\ell+1))\frac{\sigma_1^{s-1}(n+\ell+1)}{(s-1)!}\\
&\leq \sigma(n+m)
\sum_{\ell=n+1}^\infty\frac{\sigma_1^s(n+\ell)-\sigma_1^s(n+\ell+1)}{s!}=
\sigma(n+m)\frac{\sigma_1^{s}(2n+1)}{s!}.
\end{align*}
Inequality
\eqref{Jest3J} implies
\beq\label{Jfnm}
|F(n,m)|\leq \sigma(n+m)\exp\{\sigma_1(2n+1)\}\leq C_1(n)\sigma(n+m),\quad m>n,
\eeq
where $C_1(n)$ is of the same type as $C_+(n)$ in \eqref{chto}.

For $n=m$, we use the symmetry of the kernel $F(n,m)$ and
\eqref{J0.13}, \eqref{Jdef1J}, \eqref{Jest3J}, and \eqref{Jmarn}.
Then
\begin{align*}
|F(n,n)|&\leq
\bigg|\prod_{\ell=n}^\infty\frac{a(\ell)^2}{\hat a(\ell)^2} - 1\bigg| +
\sum_{\ell=n+1}^\infty \sigma^2(n+\ell) \\
&\leq C(e^{C\sigma(2n)} -1) + \sigma(2n+1)\sigma_1(2n+1),
\end{align*}
that is,
\beq\label{JdiagJ}
|F(n,n)|\leq C_2(n) \sigma(2n),
\eeq
where $C_2(n)$ again is a function similar to $C_+(n)$ in \eqref{chto},
since $\sigma_1(n)$ increases for $n\rightarrow - \infty$.
This yields \eqref{Jfourest} if we define
$p_+(\ell)=|a(\ell) - \hat a(\ell)|+ |b(\ell) - \hat b(\ell)|$ and $C_+(n)=\max\{C_1(n), C_2(n)\}$.

{\bf (ii)} Let us first show that \eqref{Jdif1} are satisfied for
$q=2,3,\dots$. By \eqref{J0.13} we have
\begin{align}
 \label{Jbplus}b(n)-b^+(n)&=a^+(n)\kappa_+(n,n+1)
  - a^+(n-1)\kappa_+(n-1,n),\\ \label{jminus}
b(n)-b^-(n)&=a^-(n-1)\kappa_-(n,n-1) -
a^-(n)\kappa_-(n+1,n).
\end{align}
Consider the "$+$" case of \eqref{Jmaro}, multiply it at $(n,n+1)$ by $n^q a^+(n)$
and at $(n-1,n)$ by $n^q a^+(n-1)$, and subtract these two equations from each other.
Then using \eqref{Jbplus}, \eqref{Jdef1J}, \eqref{Jfnm}, and monotonicity of
$\si(n)$ we obtain
\begin{align} \nn
&\sum_{n = n_0}^{+ \infty}|n|^q \big|a^+(n) F_{+}(n,n+1) - a^+(n-1) F_{+}(n - 1, n)\big| \\ \nn
& \leq \sum_{n = n_0}^{+ \infty}|n|^q \left|b(n)-b^+(n)\right| +
2C_1(n_0)\Big(\max_{j\in\Z}a^+(j)\Big)\sum_{n = n_0}^{+ \infty}
\sum_{\ell=n}^{+\infty} n^q\si(n-1+\ell)\si(n+\ell)\\  \nn
&\leq \tilde C(n_0)\bigg(1 + \sum_{n = n_0}^{+ \infty}n\si(2n) \sum_{\ell=n_0}^{+\infty}
(n+\ell)^{q-1}\si(n+\ell)\bigg)<\infty. \\ \label{Jdiff}
\end{align}
Note that for $q=1$ the last inequality in \eqref{Jdiff} does not hold, since
$n\sigma(n)\notin\ell^1(\mathbb Z_+)$. It is for this reason that we multiply by
$n^0$ in ${\bf (ii)}$.
The "$-$" case of \eqref{Jdif1} is obtained analogously for $q=2,3,\dots$,
using \eqref{jminus}, \eqref{Jdef1J}, \eqref{Jfnm}, and the symmetry
$F_-(n,m)=F_-(m,n)$.
To prove \eqref{Jfourest1}, we use \eqref{J0.13} and \eqref{Jhypo} to obtain
\begin{align}\nn
\left|\frac{1}{K_+(n,n)^2} -
\frac{1}{K_+(n+1, n+1)^2}\right|&\leq\frac{1}{a_+(n)^2}
\prod_{j=n+1}^{+\infty}\frac{a(j)^2}{a^+(j)^2}|a^+(n)^2 - a(n)^2|\\ \nn
&\leq |a(n) - a^+(n)|\,e^{2\sigma_1(n)} \sup_{j\in\Z}\frac{|a(j) + a^+(j)|}{a^+(j)^2}.
\end{align}
Using this estimate and \eqref{Jmarn} we get \eqref{Jfourest1} by the same reasoning as
in \eqref{Jdiff}.
\end{proof}

\begin{theorem}\label{Jtheor1}
Let $H\in\mathcal B_q(H^+,H^-)$ for a fixed  $q=1,2,\dots$.
Then the scattering data $\mathcal S$ of $H$ given by \eqref{J4.6}
satisfy properties ${\bf I}$--${\bf III}$ of Lemma~\ref{Jlem2.3}.
The functions $F_\pm(n,m)$ defined by (\ref{kerM}) satisfy condition
${\bf IV_q}$ of Lemma~\ref{Jlem4.1}.
\end{theorem}

It turns out, that conditions ${\bf I}$--${\bf IV_q}$ are necessary and sufficient for solving
the direct/inverse problem in the class $\mathcal B_q(H^+,H^-)$ with $q=2,3,\dots$.
For $q=1$, they are sufficient to solve the inverse scattering problem in
$\mathcal B_0(H^+, H^-)$. In the next section we will prove the sufficiency of these
conditions and describe the solution algorithm.

\section{Inverse scattering problem}
\label{secINV}

{\it Step 6}.
Let $H^\pm$ be two finite-gap Jacobi operators associated with $a^\pm(n)$, $b^\pm(n)$.
Correspondingly, $\si_+$, $\si_-$, $\psi^\pm$, and $\tilde\psi^\pm$ are given.
Let $\mathcal{S}$ be a set of the form \eqref{J4.6}, which satisfies conditions
${\bf I}$--${\bf III}$, and introduce $F_\pm(n,m)$ by (\ref{kerM}).
It follows from ${\bf I}$--${\bf III}$ that $F_\pm(n,m)$ are well defined.
We now assume that $F_\pm(n,m)$ satisfy property ${\bf IV_q}$, ${\bf (i)}$ for a fixed
$q=1,2,\dots$ and show that the Marchenko equations (\ref{Jglm1 q}) are uniquely
solvable with respect to $K_\pm(n,m)$.

\begin{lemma}\label{Jlem5.4}
Under condition ${\bf IV_q}$, ${\bf (i)}$ for a fixed $q=1,2,\dots$, equations (\ref{Jmaro})
have unique solutions $\kappa_\pm(n,\cdot)\in \ell^1(n\pm 1,\pm\infty)$ satisfying
\beq \label{J5.100}
|\kappa_\pm(n,m)|\leq \hat
C_\pm(n)\sum_{j=n+m}^{\pm\infty} p_\pm(j), \quad\pm m>\pm n. \eeq
Here $p_\pm(n)\geq 0$ are the same functions as in ${\bf IV_q}$, ${\bf (i)}$, and $\hat C_\pm(n)>0$
are bounded as $n\to\pm\infty$.
\end{lemma}

\begin{proof} Equations (\ref{Jmaro}) are generated by compact operators, hence for unique solvability
it is sufficient to prove that the homogeneous equations
\beq \label{J5.102}
f(m)+\sum_{\ell=n}^{\pm\infty} F_\pm(\ell,m) f(\ell) =0
\eeq
only have trivial solutions in $\ell^1(n,\pm\infty)$.
The proof is the same for both cases, we will consider the "$+$" case. Suppose that $f(\ell)$,
$\ell > n$, is a nontrivial solution of (\ref{J5.102}) and set $f(\ell)=0$ for $\ell \leq n$.
Since $F_+(\ell,n)\in\R$, we can assume that $f(\ell)\in\R$. Denote by
\beq \label{J5.103}
\widehat f(\la) =\sum_{m \in \mathbb{Z}}
\psi^+(\la,m) f(m)
\eeq
the generalized Fourier transform which is generated by the spectral decomposition (\ref{J1.14}) (cf.\ \cite{T}). Recall that $\widehat f(\la)\in L^1_{loc} (\siu_+\cup\sil_+)$. Multiplying
(\ref{J5.102}) by $f(m)$, summing over $m \in \Z$, and using consecutively \eqref{J1.14}, (\ref{kerM}),
(\ref{J5.103}), and ${\bf I}$, {\bf (a)}, we arrive at
\begin{align}
\begin{split}
& 2 \int_{\si_+^u}|\widehat f(\la)|^2 d\omega_+(\la) + 2 \Re
\int_{\si_+^u}R_+(\la) \widehat f(\la)^2 d\omega_+(\la) \\
\label{J5.104} & \quad +\int_{\si_-^{(1),u}}\widehat f(\la)^2
|T_-(\la)|^2 d\omega_-(\la) + \sum_{k=1}^p \ga_{+,k} \bigg(\sum_{n
\in \mathbb{Z}} \tilde \psi^+(\la_k,n) f(n) \bigg)^2=0.
\end{split}
\end{align}
Since $\widehat f(\la) \in\mathbb{R}$ as $\la\in\si_-^{(1)}$ and
$\tilde \psi^+(\la_k)\in \R$, the last two summands in (\ref{J5.104}) are nonnegative.
The integrands of the first two summands are estimated by
\[
|\widehat f(\la)|^2 + \Re R_+(\la) \widehat f(\la)^2 \geq |\widehat
f(\la)|^2 - |R_+(\la)
 \widehat f(\la)^2|
\geq \big(1-|R_+(\la)|\big)|\widehat f(\la)|^2.
\]
Dropping the last summand in (\ref{J5.104}) and using
\[
\int_{\si_+^{(1),\mathrm{u}}}(1-|R_+(\la)|) |\widehat f(\la)|^2
d\omega_+(\la) = 0,
\]
which follows from ${\bf I}$, ${\bf (b)}$, yields
\beq \label{J5.105}
2\int_{\si^{(2),\mathrm{u}}} (1-|R_+(\la)|)|\widehat f(\la)|^2
d\omega_+(\la) + \int_{\si_-^{(1),\mathrm{u}}} \widehat f(\la)^2
|T_-(\la)|^2 d\omega_-(\la) \leq 0.
\eeq
Since $|R_+(\la)|<1$ and $\omega^\prime_+(\la)>0$ for $\la\in \inte(\si^{(2)})$ and
since $\omega_-^\prime(\la)>0$ for $\la\in\inte(\si_-^{(1)})$, we have
\beq \label{Jbolsh}
\widehat f(\la)=0, \quad
\la\in\si^{(2)}\cup \si_-^{(1)}= \si_-.
\eeq
The function $\widehat f(\la)$ defined by (\ref{J5.103}) is meromorphic on
$\mathbb{C}\setminus\si_+$. But \eqref{Jbolsh} implies that $\widehat f(\la)$ is in fact meromorphic on $\mathbb{C}\setminus\si^{(1)}_+$ and vanishes on $\si_-$. Hence $\widehat f(\la)=0$, therefore
$f(m)=0$ as desired, and \eqref{Jmaro} has a unique solution in $\ell^1(n+1,+\infty)$.

To obtain estimate \eqref{J5.100}, we first introduce the following operator in
$\ell^1(n,+\infty)$,
$$
\left((I+\mathcal F_n^+)f\right)(m):= f(m)+\sum_{\ell=n}^{+\infty} F_+(\ell,m)
f(\ell).
$$
We already proved that the left-hand side of \eqref{J5.104}, which is equal to the scalar product $\left((I+\mathcal F_n^+)f, f\right)$, is nonnegative. Indeed, it is positive. If the spectrum
of multiplicity two is nonempty, that is, if there exists a set where $|R_+|<1$, then
the sum of the first two summands in \eqref{J5.104} is positive and the other two summands are
nonnegative. If $\si^{(2)}=\emptyset$, then necessarily the summand which corresponds to integrating
over the set $\si_-^{(1),u}$ is positive.
Hence, the operator $(I+\mathcal F_n^+)$ is compact and positive, and has a bounded inverse
which we abbreviate by
$$
\|(I+\mathcal F_n^+)^{-1}\|=\hat C^+(n)>0.
$$
Note that $\hat C^+(n)$ is not monotonic, there exists
\beq\label{JJes}
\max_{n>n_0}\hat C^+(n)=C_0^+.
\eeq
Let $C_+(n)$ and $p_+(n)$ be the functions of (the "$+$" case of) estimate \eqref{Jfourest} and
set
\beq\label{Jest1Jp}
\sigma^+(n):=C_+(n)
\sum_{\ell=n}^{+\infty} p^+(\ell),
\quad\sigma_1^+(n):=\sum_{\ell=n}^{+\infty} \sigma^+(\ell).
\eeq
It follows from the Marchenko equation
\eqref{Jmaro} that
$$
\kappa_+(n,n+j)= -(I+\mathcal
F_{n+1}^+)^{-1}F_+(n, n+j), \quad j=1,2,\dots,
$$
that is,
$$
\sum_{j=1}^{+\infty} |\kappa(n,n+j)|\leq \hat
C^+(n)\sum_{j=1}^{+\infty}\sigma^+(2n+j)=\hat
C^+(n)\sigma_1^+(2n+1).
$$
Applying this inequality to \eqref{Jmaro}
and using \eqref{Jfourest}, \eqref{Jest1Jp}, we obtain
\begin{align*}
|\kappa_+(n,m)|& \leq |F_+(n,m)| + \sum_{j=1}^{+\infty}
|\kappa_+(n,n+j)|\sigma^+(n+j+m)\\
&\leq\sigma^+(n+m)(1+\hat C^+(n)\sigma_1^+(2n+1))\leq \hat C_+(n)\sigma^+(n+m),
\end{align*}
which proves \eqref{J5.100}.
\end{proof}

Summarizing, we constructed $\kappa_\pm(n,m)$ for any $n$ and $\pm m>\pm n$.
To complete the reconstruction of the operators $K_\pm(n,m)$, $\pm m\geq \pm n$, it is necessary to
verify that the left hand side of \eqref{Jmarn} is a positive function.
It allows us to define $K_\pm(n,n)$ as a positive value, details of the proof can be found in \cite{emtqps}.

Now we denote the following four sequences by $a_\pm$, $b_\pm$, $n\in\Z$,
\begin{align} \label{J5.1}
\begin{split}
a_+(n) &= \frac{a^+(n)K_+(n+1, n+1)}{K_+(n,n)}, \quad
a_-(n) = a^-(n)\frac{a^-(n)K_-(n, n)}{K_-(n+1,n+1)}, \\
b_+(n) &= b^+(n) + a^+(n)\kappa_+(n,n+1)
- a^+(n-1) \kappa_+(n-1,n), \\
b_-(n) &= b^-(n) + a^-(n-1)\kappa_-(n,n-1) - a^-(n)
\kappa_-(n+1,n),
\end{split}
\end{align}
and investigate their asymptotical behavior for large $|n|$.

\begin{lemma}\label{Jlemdop} Let $F_\pm(n,m)$ satisfy condition
${\bf IV_q}$ for a fixed $q=1,2,\dots$. Then the solutions of the Marchenko equation (\ref{Jglm1 q})
satisfy
\begin{align} \label{Jinvest}
\begin{split}
&\sum_{n = n_0}^{\pm \infty}|n|^\alpha \big| K_{\pm}(n,n) -
K_{\pm}(n \pm 1, n \pm 1)\big| < \infty,    \\
&\sum_{n = n_0}^{+ \infty}|n|^\alpha \big|a^+(n)  \kappa_{+}(n,n+1)
- a^+(n-1)  \kappa_{+}(n - 1, n)\big| < \infty,\\
&\sum_{n = n_0}^{- \infty}|n|^\alpha \big|a^-(n-1) \kappa_{-}(n,n-1)
- a^-(n)  \kappa_{-}(n + 1, n)\big| < \infty,
\end{split}
\end{align}
where $\alpha=q$ for $q=2,3,\dots$ and $\alpha=0$ for $q=1$.
\end{lemma}

\begin{proof} It is almost literally the same as for Lemma \ref{Jlem4.1}, {\bf (ii)}.
For $n\geq n_0$, define
$$
\si(n)=\sup_{j\geq n_0}\big(\hat C_+(j) + C_+(j)\big)
\sum_{\ell=n}^{+\infty}p_+(\ell), \quad
\sigma_{1}(n):=\sum_{\ell=n}^{+\infty} \sigma(\ell),
$$
where $p_+(\ell)$,
$\hat C_+(n)$, and $C_+(n)$ are the functions in \eqref{J5.100} and \eqref{Jfourest}.
In the same way as we derived \eqref{Jdiff} we show
\begin{align*}
&\sum_{n = n_0}^{+ \infty}|n|^\alpha \big|a^+(n)  \kappa_{+}(n,n+1)
- a^+(n-1)  \kappa_{+}(n - 1, n)\big| \\
& \leq\sum_{n = n_0}^{+ \infty}|n|^\alpha \big|a^+(n)  F_{+}(n,n+1)
- a^+(n-1)  F_{+}(n - 1, n)\big|\\
&\quad + 2\sup_{j\in\Z}a^+(j)\times\left\{ \begin{array}{ll}
 \sum_{n = n_0}^{+ \infty}n\si(2n) \sum_{\ell=n_0}^{+\infty}
(n+\ell)^{q-1}\si(n+\ell), & q=2,\dots,\\
\sigma_{1}(2n_0)^2, & q=1.
\end{array}\right.
\end{align*}
The remaining estimates in \eqref{Jinvest} are obtained similarly.
\end{proof}

Note that (\ref{Jinvest}) imply
\beq \label{J5.2}
n^\alpha\big(|a_\pm(n)-a^\pm(n)| + |b_\pm(n) - b^\pm(n)|\big)\in\ell^1(\Z_\pm).
\eeq
Therefore, our sequences $a_+(n)$, $b_+(n)$ (resp.\ $a_-(n)$, $b_-(n)$) have the desired behavior as
$n\to +\infty$ (resp.\ $n\to -\infty$), but nothing can be said about their behavior on the opposite half-axis.

{\it Step 7}. It remains to prove that $a_+(n)\equiv a_-(n)$ and $b_+(n)\equiv
b_-(n)$ under conditions ${\bf I}$--${\bf III}$, which are the same for all fixed $q=1,2,\dots$.
We already utilized conditions ${\bf I}$, ${\bf (a)}$--${\bf (c)}$, ${\bf II}$, ${\bf III}$, ${\bf (a)}$, and
${\bf IV_q}$ in the proofs of Lemmas \ref{Jlem5.4}, \ref{Jlemdop} to obtain
$a_\pm(n)$, $b_\pm(n)$ satisfying (\ref{J5.2}). The remaining conditions
${\bf III}$, {\bf (b)}, and ${\bf I}$, {\bf (d)}, play a key role for proving the uniqueness theorem.
We remark that ${\bf III}$, {\bf (b)}, and \eqref{Jestim}, \eqref{Jestim2} are the analog of the
Marchenko condition in the step-like finite-gap case.

\begin{theorem}\label{Jtheor2}
The functions defined in \eqref{J5.1} coincide, $a_+(n)\equiv a_-(n)=a(n)$, $b_+(n)\equiv
b_-(n)=b(n)$. Moreover, the set ${\mathcal S}$, which satisfies conditions
${\bf I}$--${\bf IV_q}$, is the set of scattering data for
$H\in\mathcal B_\alpha(H^+, H^-)$ associated with the reconstructed coefficients $a(n)$, $b(n)$, where
$\alpha=0$ for $q=1$ and $\alpha=q$ for $q=2,3,\dots$.
\end{theorem}

This theorem was proven in \cite{emtstp2} by the Marchenko approach.
In the remaining part of this section we briefly discuss how the direct/inverse scattering problem
can be solved in the class $\mathcal B_1(H^+_{const}, H^-_{const})$,
where $H^\pm_{const}$ are the Jacobi operators with constant coefficients defined in \eqref{Jconstb}.
So far, results are only known for $\mathcal B_2(H^+_{const}, H^-_{const})$, see \cite{vdo}.
For $\mathcal B_1(H^+_{const}, H^-_{const})$, the characteristic properties listed in
Lemma \ref{Jlem2.3} can be simplified. First of all, in this case the background Weyl solutions
have no poles and hence $\hat\delta_\pm(\la)$, $\tilde\delta_\pm(\la)$ are obsolete.
After the change of variables
\beq\label{chanv}
\la - b^\pm=a^\pm(z_\pm + z_\pm^{-1}),
\eeq
condition ${\bf II}$ is given by

\vskip 0.2cm
 \noindent ${\bf II_{const}}$. {\it The functions $T_\pm(\la)$
can be continued analytically to $\mathbb{C} \setminus (\si\cup\si_d)$. They
satisfy $T_+(\infty)=T_-(\infty)>0$ and
$$
a^+(z_+ -z_+^{-1})T_+^{-1}(\la) = a^-(z_- - z_-^{-1})T_-^{-1}(\la) =W(\la).
$$
The function $W(\la)$ is holomorphic in $\mathbb{C}\setminus\si$ and continuous up to the boundary.
Moreover, $\overline{ W(\lau)}= W(\lal)$ for $\la\in\si$ and
$W(\la)\in\mathbb{R}$ for $\la\in\mathbb{R}\setminus \sigma$.
It has simple zeros at the points $\la_k$, where $(W^\prime(\la_k))^{-2} =\ga_{+,k}\ga_{-,k}$.
It can have simple zeros on $\Sigma_v:=\pa\si\cup(\pa\si_+^{(1)}\cap\pa\si_-^{(1)})$, but does
not vanish at other points of $\si$. If $W(E)=0$ for
$E\in\Sigma_v$, then }
\beq \label{Jestim12}
W(\la) = C\sqrt{\la - E}(1+o(1)),\quad C\neq 0, \quad \mbox{for $\la \to E$}.
\eeq

Property \eqref{Jestim12} is new and was never noticed for perturbations with finite first moments.
We will prove it in Section \ref{Contin}. As we see, the behavior of the Wronskian on constant background
can be better controlled as on arbitrary finite-gap backgrounds. In fact, \eqref{Jestim12} shows that
the Wronskian has the same behavior for first moments as for higher moments
(see remarks to Lemma \ref{Jlem2.3}) in this case.

For constant backgrounds property ${\bf III}$ is replaced by

\vskip 0.2cm
\noindent ${\bf III_{const}: (a)}$
{\it The reflection coefficients $R_\pm(\la)$ are continuous functions on $\sipmul$;
${\bf (b)}$ If $E\in\pa\si_+\cap\pa\si_-$ and $W(E)\neq 0$, then $R_\pm(E)=-1$.}

\vskip 0.2cm

The difference between ${\bf III_{const}}$ and ${\bf III}$ arises from the fact that
$R_\pm(\la)$ are continuous at the points of $\Sigma_v$ in the resonance case ($W(E)=0$) too.
For a proof of ${\bf III_{const}}$ see Section \ref{Contin}.

Now the kernels of the Marchenko equations are given by $F_\pm(n,m)=F_\pm(n+m)$, where
\begin{align*}
F_\pm(n)&=\frac{1}{2\pi\I a^\pm}\int_{|z_\pm|=1}R_\pm(\la)z_\pm^{\pm n}(z_\pm -
z_\pm^{-1})^{-1}d\la\\
& \quad +\frac{1}{2\pi\I a^\mp}\int_{\si_\mp^{(1)}}|T_\mp(\la)|^2 z_\pm^{\pm n}(z_\mp -
z_\mp^{-1})^{-1}d\la +\sum_{k=1}^q \gamma_{\pm,k} z_{\pm,k}^{\pm n}.
\end{align*}
The estimates on the transformation operator for constant backgrounds were announced in \cite{gu}
and rigorously proven in \cite[Ch.\ 10]{tjac}. They
allow to derive from the Marchenko equations that
\beq\label{Jcon}
\qquad \sum_{n=1}^{\pm\infty} |n||F_\pm(n\pm 2) -
F_\pm(n)|<\infty.
\eeq
It turns out (cf.\ \cite{tjac}), that these estimates are
sufficient to reconstruct a solution of the inverse scattering problem which
belongs to $\mathcal B_1(H^+_{const}, H^-_{const})$.
In summary,

\begin{theorem}\label{stepconst} Conditions ${\bf I}$, ${\bf II_{const}}$, ${\bf III_{const}}$, and
\eqref{Jcon} for the scattering data \eqref{J4.6} are necessary and sufficient for solving the
direct/inverse scattering problem in the class $\mathcal B_1(H^+_{const}, H^-_{const})$.
\end{theorem}

\section{On the continuity of the reflection coefficients}
\label{Contin}
The continuity of the reflection coefficient at the edge of the continuous spectrum in the resonance case has a long history. This question arose around 1975 in an attempt to clarify some characteristic properties, namely the Marchenko condition, for the scattering data of the Schr\"odinger equation on the
whole axis with a fast decaying potential of finite first moment. For higher moments $q$ starting from $q=2$, the answer is evident. For such $q$, the Jost solutions and their conjugates are differentiable
from the side of the spectrum with respect to the local parameter $\sqrt{\la - E}$, where $E$ ($E=0$ for the decaying potential) is the edge of the continuous spectrum. Therefore, their Wronskians are also differentiable with respect to the local parameter, and to prove continuity of the reflection coefficient in the resonance case it is sufficient to use de l'Hopital. This situation complicates a lot for finite first moments,
because the Jost solutions are no longer differentiable at $E$. The problem for first moments was solved independently by Gusseinov \cite{Gus2} and Klaus \cite{Kl} (for a revised version see \cite{AK}).

In this section, we combine methods developed in \cite{Gus2}, \cite{AK}
and adapt them to the discrete model to prove \eqref{Jestim12} and the continuity of the reflection coefficient
first for the Jacobi operator $H:\lz\to\lz$ with real coefficients $a(n)>0$, $b(n)$ satisfying
\beq\label{first}
\sum_{n\in\Z}|n|\Big(\Big|a(n)-\frac 1 2 \Big|+
|b(n)|\Big)<\infty.
\eeq
This result is of independent interest. Then we prove \eqref{Jestim12}
and ${\bf III_{const}}$, ${\bf (a)}$ for arbitrary constant backgrounds
which completes the proof of Theorem \ref{stepconst}.

The Jacobi spectral equation with coefficients satisfying \eqref{first},
\begin{equation}\label{main}
 a(n) y(z,n+1) + a(n-1) y(z,n-1) + b(n) y(z,n)=
 \frac{1}{2}(z+z^{-1})y(z,n),\quad |z|\leq 1,
\end{equation}
has the following two Jost solutions,
\beq\label{Jost}
\vp_\pm(z,j)=\sum^{\pm\infty}_{\ell=j} K_\pm(j,\ell)z^{\pm \ell}.
\eeq
Recall that the kernels $K_\pm(j,\ell)$ of the transformation operators satisfy
\beq\label{trans}
|K_\pm(j,\ell)|\leq
C\sum^{\pm\infty}_{n=\left[\frac{j+\ell}{2}\right]}\Big(\Big|a(n)-\frac 1 2 \Big|+
|b(n)|\Big)
\eeq
in this case. The functions $C_\pm(\cdot)$ in \eqref{chto} which decay in one direction and grow in
the other direction are replaced by a constant $C$ here
(and in the general case of coinciding finite-gap backgrounds $H^+=H^-$ as well). It follows from
\eqref{first}, \eqref{trans} that $\vp_\pm(z)$ are continuous for $z\in\mathbb{T}:=\{z:|z|=1\}$.
In general, they are not differentiable with respect to $z$ at $\hat z=+1,-1$, since
$\sum^{\pm\infty}_{\ell=j} |\ell K_\pm(j,\ell)|$ diverge.
Let
$$
W(z)=W(\vp_-,\vp_+)=a(0)(\vp_-(z,0)\vp_+(z,1) -
\vp_-(z,1)\vp_+(z,0))
$$
be the Wronskian of the Jost solutions and define
$$
W_1(z)=a(0)(\vp_-(z,0)\overline{\vp_+(z,1)} -
\vp_-(z,1)\overline{\vp_+(z,0)}).
$$

\begin{lemma}\label{lem11}
For $\hat z\in\{1,-1\}$, suppose that $W(\hat z)=0$. Then the following representations are valid
\begin{align}\label{est1}
W(z)&=C(z-\hat z)(1+o(1)) \quad \mbox{as $z\to\hat z$}, \\ \label{est2}
W_1(z)&=C_1(z-\hat z)(1+o(1)) \quad \mbox{as $z\to\hat z$ and $z\in\mathbb T$},
\end{align}
where $C\neq 0$ and $C_1$ are constants.
\end{lemma}
\begin{proof}
We will use summation by parts, i.e., the following identity,
\beq\label{int}
\sum^{\pm\infty}_{\ell=s} \big(f(\ell) - f(\ell\pm 1)\big)v(\ell) = \sum^{\pm\infty}_{\ell=s} f(\ell)
\big(v(\ell) - v(\ell\mp 1)\big) + f(s)v(s\mp 1),
\eeq
which is valid for all $f(\cdot)\in\liz$, $\sup_{\ell\in\mathbb{Z}_\pm}|v(\ell)|<\infty$ or vice versa.

Introduce
\beq\label{defPh}
\Phi_\pm^{(j)}(s)=\sum^{\pm\infty}_{\ell=s}
K_\pm(j,\ell)\hat z^{\ell}
\eeq
which are well defined due to \eqref{trans} and \eqref{first}.
Moreover, since $\hat z^{-1}=\hat z$,
\beq\label{main2}
\Phi_\pm^{(j)}(j)=\vp_\pm(\hat z,j).
\eeq
Applying \eqref{int} to \eqref{Jost} we obtain
\begin{align*}
&\vp_\pm(z,j)=\sum_{\ell=j}^{\pm\infty}\Phi_\pm^{(j)}(\ell)\big((\hat z z)^{\pm \ell}
- (\hat z z)^{\pm \ell-1}\big)+\Phi_\pm^{(j)}(j )(\hat z z)^{\pm j-1}\\
&=\vp_\pm(\hat z, j) +\Phi_\pm^{(j)}(j)\left((\hat z z)^{\pm j-1}- \hat z^{\pm
2j-2}\right)+\sum_{\ell=j}^{\pm\infty}\Phi_\pm^{(j)}(\ell)(\hat z z)^{\pm
l}\left(1-(\hat z z)^{-1}\right).
\end{align*}
Abbreviate
$$
\zeta(z):=\frac{z-\hat z}{z},
$$
then
\begin{align}\label{phi0}
\vp_\pm(z,0)&=\zeta(z) \sum_{\ell=\frac{1\pm 1}{2}}^{\pm\infty}\Phi_\pm^{(0)}(\ell)(\hat z z)^{\pm \ell}
+(z\hat z)^{\frac{\pm 1 - 1}{2}}\vp_\pm(\hat z, 0), \\ \label{phi1}
\vp_\pm(z,1)&=\zeta(z) \sum_{\ell=\frac{1\pm 1}{2}}^{\pm\infty}\Phi_\pm^{(1)}(\ell)(\hat z z)^{\pm \ell}
+(z\hat z)^{\frac{\pm 1 - 1}{2}}\vp_\pm(\hat z, 1).
\end{align}
Multiplying \eqref{phi0} by $\vp_\pm(\hat z,1)$ and \eqref{phi1} by
$\vp_\pm(\hat z,0)$, their difference is equal to
\beq\label{hatW}
\breve W_\pm(z):=\vp_\pm( z,1)\vp_\pm(\hat z,0)- \vp_\pm( z,0)\vp_\pm(\hat
z,1)=\zeta(z)\Psi_\pm(z),
\eeq
where
\begin{align} \nn
\Psi_\pm(z)&:=\sum_{\ell=\frac{1\pm 1}{2}}^{\pm\infty}h_\pm(\ell)(\hat z z)^{\pm \ell},\\ \label{defh}
h_\pm(\ell)&:=\Phi_\pm^{(1)}(\ell)\vp_\pm(\hat z,0) -\Phi_\pm^{(0)}(\ell)\vp_\pm(\hat z,1).
\end{align}
Note that by \eqref{defPh}, \eqref{trans} we have $h_\pm(\cdot) \in \ell^\infty(\mathbb{Z}_\pm)$.
Suppose that we can show
\beq\label{el1}
h_\pm(\cdot)\in\liz.
\eeq
Then $\Psi_\pm(z)$ are continuous at $\hat z$ and \eqref{hatW} implies
\begin{align}\label{deriv}
\frac{\vp_\pm(z,1)}{\vp_\pm(z,0)}&=\frac{\vp_\pm(\hat
z,1)}{\vp_\pm(\hat z,0)} +C_0^\pm(z-\hat z)(1 + o(1)),\quad \mbox{as $\vp_\pm(\hat z,0)\neq 0$}, \\ \label{deriv1}
\frac{\vp_\pm(z,0)}{\vp_\pm(z,1)}&=\frac{\vp_\pm(\hat
z,0)}{\vp_\pm(\hat z,1)} +C_1^\pm(z-\hat z)(1 + o(1)),\quad \mbox{as $\vp_\pm(\hat z,1)\neq 0$},
\end{align}
for constants $C_1^\pm$, $C_0^\pm$.
Both representations hold for all $|z|\leq 1$ in a vicinity of $\hat z$.

Hence, if $W(\hat z)=0$, then $\vp_-(\hat z,j)=\gamma\vp_+(\hat z,j)$ and we
have two possibilities, either
\beq\label{a}
\vp_-(\hat z,0)\vp_+(\hat z,0)\neq 0,
\eeq
or
\beq\label{b}
\vp_-(\hat z,0)=\vp_+(\hat z,0)=0.
\eeq
In the first case we obtain from \eqref{deriv}
\begin{align}\label{wrm} \nn
W(z)&=a(0)\vp_-(z,0)\vp_+(z,0)\Big(\frac{\vp_+(\hat z,1)}{\vp_+(\hat z,0)}
- \frac{\vp_-(\hat z,1)}{\vp_-(\hat z,0)}+(C_0^+-C_0^-)(z-\hat z)(1+o(1))\Big)\\
&=a(0)\vp_-(\hat z,0) \vp_+(\hat z,0)(C_0^+ -C_0^-)(z-\hat z)(1+o(1)),
\end{align}
which implies \eqref{est1} with $C=a(0)\vp_-(\hat z,0) \vp_+(\hat z,0)(C_0^+
-C_0^-)$. It follows from the general estimate \eqref{Jestim} that
$(z^2-1)/W(z)$ is bounded for $z\in\mathbb T$ in small vicinities of $+1,-1$, therefore
$C\neq 0$.
For the case \eqref{b} we use \eqref{deriv1}.
To obtain \eqref{est2}, we substitute $z^{-1}$ for $z$ (which is possible as $z\in\mathbb T$)
in the "$+$" cases of \eqref{Jost}, \eqref{main2}--\eqref{deriv1} and also use $\overline{\hat z}=\hat z$.

Therefore, in order to finish the proof it remains to show \eqref{el1}. Let us first consider
$h_+$. For $m\geq 1$, the transformation operators involved in \eqref{el1} satisfy
the following Marchenko equations,
\begin{align*}
&K_+(0,m)+ \sum_{\ell=0}^{+\infty} K_+(0,\ell)F_+(\ell+m)=0, \\
&K_+(1,m) +\sum_{\ell=1}^{+\infty} K_+(1,\ell)F_+(\ell+m)=\frac{\delta(1,m)}{K_+(1,1)}.
\end{align*}
Multiplying both equalities by $\hat z^m$ and summing from $m=s\geq 1$ to $+\infty$
gives
\begin{align*}
&\Phi_+^{(0)}(s)+\sum_{m=s}^{+\infty}\sum_{\ell=0}^{+\infty}F_+(\ell+m)\hat
z^{\ell+m} \left(\Phi_+^{(0)} (\ell)-\Phi_+^{(0)}(\ell+1)\right)=0, \\
&\Phi_+^{(1)}(s)+\sum_{m=s}^{+\infty}\sum_{\ell=1}^{+\infty}F_+(\ell+m)\hat
z^{\ell+m} \left(\Phi_+^{(1)}(\ell)-\Phi_+^{(1)}(\ell+1)\right)=\frac{\delta(1,s)}{K_+(1,1)}.
\end{align*}
We set $v(\ell)=F_+(\ell)\hat z^\ell$ and sum by parts according to \eqref{int},
\begin{align*}
&\Phi_+^{(0)}(s)+\sum_{m=s}^{+\infty}\bigg(\sum_{\ell=0}^{+\infty}
\left(v(\ell+m)- v(\ell+m-1)\right)\Phi_+^{(0)} (\ell) +\Phi_+^{(0)} (0)v(m-1)\bigg)=0, \\
&\Phi_+^{(1)}(s)+\sum_{m=s}^{+\infty}\bigg(\sum_{\ell=1}^{+\infty} \left(v(\ell+m)-
v(\ell+m-1)\right)\Phi_+^{(1)} (\ell) + \Phi_+^{(1)}(1)v(m)\bigg)=\frac{\delta(1,s)}{K_+(1,1)}.
\end{align*}
Summing over $m$ and taking \eqref{main2} into account yields
\begin{align}\label{4}
\Phi_+^{(0)} (s) + \vp_+(\hat z,0)\sum_{m=s}^{+\infty} v(m)
-\sum_{\ell=1}^{+\infty}\Phi_+^{(0)} (\ell)v(\ell+s-1)&=0, \\ \nn
\Phi_+^{(1)} (s) + \vp_+(\hat z,1)\sum_{m=s}^{+\infty} v(m)
-\sum_{\ell=1}^{+\infty}\Phi_+^{(1)} (\ell)v(\ell+s-1)&=\frac{\delta(1,s)}{K_+(1,1)}.
\end{align}
We again have to distinguish two cases. If \eqref{a} holds, we multiply the first equation by
$\vp_+(\hat z,1)$, the second by $\vp_+(\hat z,0)$, subtract the
first from the second, and use \eqref{defh} to arrive at
\beq\label{main3}
h_+(s) -\sum_{\ell=1}^{+\infty}
h_+(\ell)v(\ell+s-1)=\frac{\delta(1,s)}{K_+(1,1)}\vp_+(\hat z,0).
\eeq
We know that $h_+(\cdot)\in \ell^\infty(\mathbb{Z}_+)$, and from the estimates given in
\cite{gu} and \cite[Ch.\ 10]{tjac} we have (recall that $|v(\ell)|= |F_+(\ell)|$)
 \beq\label{vest}
|v(\ell)|\leq C\sum^{+\infty}_{n=\left[\frac{\ell}{2}\right]}\Big(\Big|a(n)-\frac 1 2 \Big|+
|b(n)|\Big).
\eeq
But any bounded solution of \eqref{main3} with a kernel satisfying \eqref{vest} already is in
$\ell^1(\mathbb Z_+)$, as proved in \cite{mar} (see (3.2.24)--(3.2.25)).
For the case \eqref{b} we have
$$
\vp_+(\hat z,0)=0, \quad h_+(s)=\vp_+(\hat z,1)\Phi_+^{(0)}(s).
$$
Therefore, \eqref{4} implies
$$
h_+(s) -\sum_{\ell=1}^{+\infty} h_+(\ell)v(\ell+s-1)=0,
$$
and hence again $h_+(\cdot)\in \ell^1(\mathbb{Z}_+)$.
\end{proof}
The continuity of the reflection coefficients $R_\pm$ follows from Lemma~\ref{lem11},
since $R_+(z)=-W_1(z)W^{-1}(z)$, $R_-(z)=\overline{W_1(z)}W^{-1}(z)$.

Now we are ready  to prove \eqref{Jestim12} for the case of two different constant background operators  $H^\pm_{const}$, defined by \eqref{Jconstb}, and $q=1$. The spectra $\si_\pm$ of $H^\pm_{const}$
consist of intervals,
\beq\label{gran}
\si_\pm=[c^\pm, d^\pm], \quad \mbox{where $c^\pm:=b^\pm-2a^\pm$, $d^\pm:=b^\pm+2a^\pm$}.
\eeq
Let us first specify the set $\Sigma_v$ of possible virtual levels for this case. In order to
complete the proof of property ${\bf III_{const}}$, ${\bf (a)}$, compared to
${\bf III}$, ${\bf (a)}$, we have to show continuity of one or the other reflection coefficient
on $\Sigma_v$. The following mutual locations of the spectra $\si_\pm$ are possible.
\begin{enumerate}[(i)]
\item $\si_-\subset\si_+$ (or vice-versa). Then $\Sigma_v=\pa\si_+$, and we have to show continuity of $R_+$;
\item $d^+\leq c^-$ (or $d^-\leq c^+$). Then $\Sigma_v=\{c^+,d^+, c^-, d^-\}$, and we have to show
continuity of $R_\pm$ at $c^\pm$ and $ d^\pm$;
\item $\si^{(2)}\neq\emptyset$, for example,
$\si^{(2)}=[c^-, d^+]$. Then $\Sigma_v=\{c^+, d^-\}$, and we have to show continuity of $R_+$ at $c^+$ and of $R_-$ at $d^-$.
\end{enumerate}
We start with case (i), $\Sigma_v=\{c^+, d^+\}$. Suppose that $E=d^+$, $W(E)=0$, and introduce the change of variables $\la-b^+=a^+(z+z^{-1})$ (cf.\ \eqref{chanv}). The point $d^+$ is mapped to
$\hat z=1$ and the Jost solution $\phi_+(\la)=\varphi_+(z)$ is now considered as a function of $z$
on the unit circle with representation \eqref{Jost}. Repeating literally the beginning of the proof of Lemma \ref{lem11} we obtain \eqref{hatW} and \eqref{defh} for $\varphi_+(z)$. Then
\eqref{deriv} and \eqref{deriv1} follow from \eqref{el1}. Observe that the only structural information on
the kernel of the Marchenko equation which was used to prove \eqref{el1} is inequality \eqref{vest}. But
\eqref{vest} holds in the steplike case too, since we use it for $\ell\geq 0$ where $C_1(n)$
(see \eqref{Jfnm}, \eqref{chto}) can be replaced by a constant.

Hence \eqref{deriv} and \eqref{deriv1} are valid in the steplike case, but only for the Jost solution for
which the point $\hat z$ is an edge of the spectrum of its background operator. In turn, \eqref{deriv} and \eqref{deriv1} imply for all $\la$ in a vicinity of $E=d^+$,
 \begin{align}\label{derivp}
\frac{\phi_+(\la,1)}{\phi_+(\la,0)}
&=\frac{\phi_+(E,1)}{\phi_+(E,0)} +C_0^+\sqrt{\la - E}(1 + o(1)), \quad \mbox{as
$\phi_+(E,0)\neq 0$}, \\ \label{derivp1}
\frac{\phi_+(\la,0)}{\phi_+(\la,1)}&=
\frac{\phi_+(E,0)}{\phi_+(E,1)} +C_1^+ \sqrt{\la - E}(1 + o(1)), \quad \mbox{as
$\phi_+(E,1)\neq 0$},
\end{align}
where $C_1^+$, $C_0^+$ are constants. Evidently, the same expressions will hold for $\frac{\overline{\phi_+(\la,1)}}
{\overline{\phi_+(\la,0)}}$ (resp.\ $\frac{\overline{\phi_+(\la,0)}}
{\overline{\phi_+(\la,1)}}$) as $\la\to E$ and $\la<E$.

On the other hand, if $d^-<d^+$ then $\phi_-(\la,\cdot)$ are $C^\infty$-functions of $\la$,
and we have
\begin{align}\label{derivm}
\frac{\phi_-(\la,1)}{\phi_-(\la,0)}
&=\frac{\phi_-(E,1)}{\phi_-(E,0)} +\hat C_0^-(\la - E)(1 + o(1)), \quad \mbox{as
$\phi_-(E,0)\neq 0$}, \\ \label{derivm1}
\frac{\phi_-(\la,0)}{\phi_-(\la,1)}&=
\frac{\phi_-(E,0)}{\phi_-(E,1)} +\hat C_1^- (\la - E)(1 + o(1)), \quad \mbox{as
$\phi_-(E,1)\neq 0$},
\end{align}
for constants $\hat C_1^-$, $\hat C_0^-$.
If $d^+=d^-$, then upon the change of variables $\la-b^-=a^-(z+z^{-1})$,
where $d^-=d^+=E$ corresponds to $\hat z =1$, we obtain \eqref{deriv}, \eqref{deriv1}
for $\vp_-(z,\cdot)$. Therefore,
\begin{align}\label{derivmm}
\frac{\phi_-(\la,1)}{\phi_-(\la,0)}
&=\frac{\phi_-(E,1)}{\phi_-(E,0)} + C_0^-\sqrt{\la - E}(1 + o(1)), \quad \mbox{as
$\phi_-(E,0)\neq 0$}, \\ \label{derivmm1}
\frac{\phi_-(\la,0)}{\phi_-(\la,1)}&=
\frac{\phi_-(E,0)}{\phi_-(E,1)} + C_1^-\sqrt{\la - E}(1 + o(1)), \quad \mbox{as
$\phi_-(E,1)\neq 0$},
\end{align}
for all $\la$ in a small vicinity of $E$.
The same considerations as for \eqref{wrm} imply now for the case \eqref{a} that
$W(\la)=C\sqrt{\la - E}(1+o(1))$ with $C=-a(0)\phi_-(\la,0)\phi_+(\la,0)C_0^+$ if $d^-<d^+$, and $C=a(0)\phi_-(\la,0)\phi_+(\la,0)(C_0^- - C_0^+)$
if $d^-=d^+$. Due to \eqref{Jestim}, $C\neq 0$, which proves \eqref{Jestim12} in this case.
It is clear that for
$$
W_1(\la)=a(0)(\phi_-(\la,0)\overline{\phi_+(\la,1)} -\phi_-(\la,1)\overline{\phi_+(\la,0)}
$$
the representation $W_1(\la)=C_1\sqrt{\la - E}(1+o(1))$ from the side of the spectrum will hold,
but here we cannot check whether $C_1$ vanishes or not. This representation and \eqref{Jestim12}
imply the continuity of the reflection coefficient $R_+$ for case (i).
Cases (ii) and (iii) can be treated in the same way, which finishes the proof
of Theorem~\ref{stepconst}.

\bigskip
\noindent{\bf Acknowledgments.}
We thank A.\ K.\ Khanmamedov for bringing the mistake in \cite{emtqps} to our attention.
Furthermore, we are indebted to the anonymous referee for valuable suggestions improving
the presentation of the material.

\end{document}